\documentclass[leqno,12pt]{amsart}
\usepackage{amsfonts}
\usepackage{amsmath,amssymb,amsthm}
\usepackage{amsfonts}

\newtheorem{theorem}{Theorem}[section]
\newtheorem{proposition}[theorem]{Proposition}
\newtheorem{lemma}[theorem]{Lemma}
\newtheorem{corry}[theorem]{Corollary}

\theoremstyle{definition}

\theoremstyle{definition}
\newtheorem{rem}[theorem]{Remark}

\numberwithin{claim}{theorem}

\newcommand{\tr}{\mathrm{tr\,}}

\renewenvironment{proof}{\textit{Proof.}}{\hfill\ensuremath{\qed}}

\def \qed{\hfill{\hbox{$\square$}}}

\numberwithin{equation}{section}

\begin{document}
\title[Biconservative Surfaces in $L^n_1(f,c)$]{Biconservative Surfaces in Robertson-Walker Spaces}

 \author[N. Cenk Turgay]{Nurettin Cenk Turgay}
\address{Department of Mathematics, Faculty of Science and Letters, Istanbul Technical University, \.{I}stanbul, T{\"u}rk\.{I}ye}
\email{nturgay@itu.edu.tr}

\author[R. Ye\u{g}\.{i}n \c{S}en]{R\"{u}ya Ye\u{g}\.{i}n \c{S}en}
\address{Department of Mathematics, Faculty of Engineering and Natural Sciences, \.{I}stanbul Medeniyet University, \.{I}stanbul, T{\"u}rk\.{I}ye}
\email{ruya.yegin@medeniyet.edu.tr}

\begin{abstract}
In this paper, we mainly focus  on space-like PMCV surfaces in Robertson-Walker spaces. First, we derive certain geometrical properties of biconservative surfaces in the Robertson-Walker space $L^n_1(f, c)$ of arbitrary dimension. Then, we get complete local classifications of such surfaces in $L^4_1(f,0)$, $L^5_1(f,0)$ and $L^5_1(1,\pm 1)$. Finally, we proved that a space-like PMCV biconservative surface in $L^n_1(f,0)$ lies on a totally geodesic submanifold with dimension $4$ or $5$.
\end{abstract}

\subjclass[2010]{53A10(Primary), 53C42.}
\keywords{Biharmonic surfaces, parallel mean curvature vector, Robertson-Walker Spacetime}

\maketitle
\section{Introduction}

A biharmonic map $\psi:(M,g)\rightarrow (N,\widetilde{g})$ between two semi-Riemannian manifolds is characterized as a critical point of the bienergy functional defined by
$$ E_2:C^\infty(M,N)\to\mathbb{R}, \quad E_2(\psi)=\frac{1}{2}\int_{M}\widetilde{g}(\tau(\psi),\tau(\psi)) \, v_g, $$
where $v_g$ represents the volume element of $M$ and $\tau(\psi)$ denotes the tension of $\psi$ given by
$\tau(\psi)=\tr \nabla d\psi,$ \cite{Sasahara2012}. It is well established that the mapping $\phi$ is biharmonic if and only if it satisfies the Euler-Lagrange equation 
\begin{equation}\label{EulerLagrangeEq}
\tau_2(\psi):=\Delta \tau(\psi) - \tr \widetilde{R}(d\psi, \tau(\psi)) \, d\psi=0
\end{equation}
associated with the bi-energy functional, where $\tau_2(\psi)$ is said to be bi-tension of $\psi$, $\Delta$ and $\widetilde{R}$ denote the rough Laplacian acting on sections of $\psi^{-1}(TN)$ defined by
$$\Delta U=\tr  (\nabla_\cdot \nabla_\cdot U)$$
whenever $U\in\Gamma(\psi^{-1}(TN))$ and the curvature tensor field of $N$, respectively, \cite{Ji2}.

On the other hand, a mapping $\psi : (M,g) \rightarrow (N,\widetilde{g})$ that satisfies the condition
\begin{equation}\label{BiconsEq}
\widetilde g(\tau_2(\psi), d\psi) = 0,
\end{equation}
which is weaker than \eqref{EulerLagrangeEq}, is said to be a biconservative mapping. When $\psi=\phi$ is an isometric immersion, the equation \eqref{BiconsEq} simplifies to
\begin{equation}\label{EulerLagrangeEqBic}
\tau_2(\phi)^T = 0,
\end{equation}
where $\tau_2(\phi)^T$ denotes the tangential part of $\tau_2(\phi)$. In this case, $M$ is said to be a biconservative submanifold of $N$.

By examining the tangential component of $\tau_2(\phi)$, we can deduce the following well-known result from \eqref{EulerLagrangeEqBic}.

\begin{proposition}\label{PROPBiconser}\cite{MOR2016JGA}
Let $\phi: (M^m, g) \to  (N^n,\widetilde{g})$ be an isometric immersion between two semi-Riemannian manifolds. Then, $\phi$ is biconservative if and only if equation 
\begin{equation}\label{tangent component}
m \, \nabla \|H\|^2 + 4 \, \tr  \, A_{\nabla^\perp_{\cdot} H} (\cdot) + 4 \, \tr (\widetilde{R}(\cdot, H) \cdot)^T = 0,
\end{equation} 
is satisfied.
\end{proposition}

The study of biharmonic maps between warped products was initiated by Balmu\c{s} \textit{et. al.} in 2007, \cite{BalmusMontaldoOniciuc2007}. In the following years, Roth obtained sufficient and necessary conditions for biharmonic submanifolds of Cartesian products of two non-flat Riemannian space-forms spheres in \cite{RothUpadhyay2017}. The study of biconservative surfaces in product spaces was started by Fetcu \textit{et. al.} in \cite{FetcuOniciucPinheiro2015}, where the main focus lies on surfaces in $\mathbb{S}^n\times \mathbb{R}$ and $\mathbb{H}^n\times \mathbb{R}$ with parallel mean curvature vector (PMCV). In \cite{ManfioNCTUpadhyay2019}, this investigation extended to higher-dimensional submanifolds in $\mathbb{S}^n\times \mathbb{R}$ and $\mathbb{H}^n\times \mathbb{R}$.

On the other hand,  PMCV surfaces of different Lorentzian spaces were researched in  \cite{AnciauxCipriani2020, ChenJVanderVeken2007, ChenJVanderVeken2010} and  some open problems about PMCV surfaces in Robertson-Walker spaces were presented in \cite{DekimpeJVanderVeken2020}.
In this paper, we study space-like PMCV surfaces embedded in the Lorentzian warped product spaces. In particular,  we complete the local classification of  space-like biconservative PMCV surfaces in  $L^n_1(f,0)$, $\mathbb E^1_1\times\mathbb S^n$ and  $\mathbb E^1_1\times\mathbb H^n$ by proving Theorem \ref{THML41f0SPL}, Theorem \ref{THML51f0SPL}, Theorem \ref{THMRxRcSP1}, Proposition \ref{NonExistenceE11H4} and Theorem \ref{CoDimMainTheorem}.

In Sect. 2, we give some basic facts and equations of the theory of submanifolds, along with the notation that we are going to use throughout the article. In Sect. 3, we explore some geometrical properties of PMCV surfaces in a Robertson-Walker space $L^n_1(f,c)$ with arbitrary dimension. Sect. 4 and Sect. 5 are devoted to surfaces of $L^4_1(f,0)$ and $L^5_1(f,0)$, respectively. Furthermore, biconservative PMCV surfaces of $\mathbb{E}^1_1 \times \mathbb{S}^4$ are studied in Sect. 6. In Sect. 7, we consider some of global properties of biconservative PMCV surfaces in $L^n_1(f,c)$. As a conclusion, we show the substantial co-dimension of a biconservative PMCV surface is either 2 or 3.


\section {Preliminaries}
In this section, we give a brief summary of basic facts and equations of the theory of submanifolds, \cite{OneillBook}.

\subsection{Basic Formul\ae\ and Definitions}
Let $R^{n-1}(c)$ denote the $n-1$ dimensional Riemannian space-form with the constant sectional curvature $c$, i.e.,
$$R^{n-1}(c)=\left\{\begin{array}{cc}
\mathbb S^{n-1}&\mbox{if }c=1,\\
\mathbb E^{n-1}&\mbox{if }c=0,\\
\mathbb H^{n-1}&\mbox{if }c=-1
\end{array}\right.$$
with the metric tensor  $g_c$ and the Levi-Civita connection $\nabla^{R^{n-1}(c)}$. If $I$ is an open interval and $f:I\to\mathbb R$ is a smooth, non-vanishing function then the Robertson Walker space  $L^n_1(f,c)$ is defined as the Lorentzian warped product $I^1_1\times_f R^{n-1}(c)$ whose metric tensor $\widetilde g$ is
$$\widetilde g=\langle\cdot,\cdot\rangle=-dt^2+f(t)^2g_c.$$
The vector field $\displaystyle\frac{\partial}{\partial t} $ is known as the comoving observer field in general relativity, \cite{ChenJVanderVeken2007}.

Let  $\Pi^1:I\times R^{n-1}(c)\to I$ and $\Pi^2:I\times R^{n-1}(c)\to R^{n-1}(c)$
denote  the canonical projections. For a given vector field $X$ in ${L}^n_1(f,c)$, we define a function $X_0$ and a vector field $\bar X$ by the orthogonal decomposition
$$X=X_0\displaystyle\frac{\partial}{\partial t} +\bar X,$$
that is,
$$X=X_0\displaystyle\frac{\partial}{\partial t} +\sum\limits_{i=1}^{n-1}X_i\frac{\partial}{\partial x_i}:=(X_0,X_1,\hdots,X_{n-1}),$$ where 
$(x_1,x_2,\hdots,x_{n-1})$ is a Cartesian coordinate system in $\mathbb R^{n-1}$ and we have 
$$X_0=-\langle\displaystyle\frac{\partial}{\partial t} ,X\rangle,\qquad \Pi_*^1(\bar X)=0.$$ First, we would like to express the Levi-Civita connection of $ L^n_1(f,c)$. 
Occasionally, by misusing terminology, we are going to put $(0,\bar X)=\bar X$.

We are going to use the following lemma which can be directly obtained from  \cite{OneillBook} (See also \cite[Lemma 2.1]{ChenJVanderVeken2007}).
\begin{lemma}\label{LemmaLn1f0LCConnect}
The Levi-Civita connection $\widetilde\nabla$ of ${L}^n_1(f,c)$ is
\begin{eqnarray}
\label{RWSTconn}
\widetilde{\nabla}_X Y&=&\nabla^0_XY+\frac{f'}{f} \left(\widetilde g(\bar X,\bar Y){\displaystyle\frac{\partial}{\partial t} }+ X_0\bar Y+Y_0\bar X\right)
\end{eqnarray}
whenever $X$ and $Y$ are tangent to ${L}^n_1(f,c)$, where  $\nabla^0$ denotes the Levi-Civita connection of the Cartesian product space ${L}^n_1(1,c)=I\times R^{n-1}(c)$.
\end{lemma}

On the other hand, the Riemann curvature tensor $\widetilde R$ of ${L}^n_1(f,c)$ is as follows:
\begin{lemma}\label{LemmaLn1f0CurvTensor}\cite{ChenJVanderVeken2007}.
The Riemannian curvature tensor $\widetilde R$ of ${L}^n_1(f,c)$ satisfies
\begin{align}
\begin{split}
\widetilde R (\frac{\partial}{\partial t} ,\bar X )\frac{\partial}{\partial t} =\frac{f''}f\bar X, \qquad& \widetilde R (\frac{\partial}{\partial t} ,\bar X )\bar Y=\frac{f''}f\langle \bar X,\bar Y\rangle\frac{\partial}{\partial t} ,\\
\widetilde R (\bar X,\bar Y )\frac{\partial}{\partial t} =0, \qquad& \widetilde R (\bar X,\bar Y )\bar Z=\frac{f'{}^2+c}{f^2} (\langle \bar Y,\bar Z\rangle \bar X-\langle \bar X,\bar Z\rangle \bar Y )
\end{split}
\end{align}
whenever $\Pi_*^1(\bar X)=\Pi_*^1(\bar Y)=\Pi_*^1(\bar Z)=0$.
\end{lemma}

\subsection{Surfaces in Robertson Walker Spaces}
Consider an oriented space-like surface $M$ in $L^n_1(f,c)$ equipped with the Levi-Civita connection $\nabla$ and metric tensor $g$. For convenience, we denote the induced connection of $L^n_1(f,c)$ also by $\widetilde\nabla$. The Gauss and Weingarten formul\ae\ are given by
\begin{eqnarray}
    \label{Gauss} \widetilde\nabla_X Y &=& \nabla_X Y + h(X,Y), \\
    \label{Weingarten} \widetilde\nabla_X \xi &=& -A_\xi(X) + \nabla^\perp_X \xi,
\end{eqnarray}
defining the second fundamental form  $ h $, the shape operator  $ A $, and the normal connection  $ \nabla^\perp $  of  $ M $, where  $ X $  and  $ Y $  are tangent vectors to  $ M $  and  $ \xi $  is a normal vector to  $ M $. Note that  $ A $  and  $ h $  are related by
\begin{eqnarray}
    \label{AhRelatedBy} 
    \widetilde g(h(X,Y),\xi) = g(A_\xi X,Y).
\end{eqnarray}
Moreover, the Codazzi equation
\begin{eqnarray}    
    \label{CodEqMostGeneral} \left(\widetilde R(X,Y)Z\right)^\perp &=& \left(\nabla^\perp_{X} h\right)(Y, Z) - \left(\nabla^\perp_{Y} h\right)(X, Z)
\end{eqnarray}
is satisfied, where the covariant derivative $\nabla^\perp h$ of the second fundamental form is defined by
$$\left(\nabla^\perp_X h\right)(Y, Z) = \nabla^\perp_X h(Y, Z) - h(\nabla_X Y, Z) - h(Y, \nabla_X Z).$$
Furthermore, since Lemma \ref{LemmaLn1f0CurvTensor} implies $\left(\widetilde R(X,Y)\xi\right)^\perp=0$, the Ricci equation turns into
\begin{eqnarray}
    \label{RicciLn1fc} 
    R^\perp(X,Y)\xi &=& h(X, A_\xi Y) - h(A_\xi X, Y),
\end{eqnarray}
where $R^\perp$ is the normal curvature tensor, \cite{ChenJVanderVeken2007}.

The mean curvature vector field $H$  of $M$  is defined by 
 $$H = \frac{1}{2} \tr h. $$
 $M$  is said to be PMCV if $H$  is parallel along the normal connection, i.e., $\nabla^\perp H = 0$. Further, the first and second normal space $N_1M$ and $N_2M$ of $M$ is defined by
\begin{eqnarray}
N_1M&=&\mathrm{span\,}\{h(X,Y)|X,Y\in TM\}\\
N_2M&=&\mathrm{span\,}\{h(X,Y),\nabla^\perp_Zh(X,Y)|X,Y,Z\in TM\},
\end{eqnarray}
respectively.

On the other hand, we decompose  $ \displaystyle\frac{\partial}{\partial t}  $  into its tangential and normal parts by
\begin{equation}
    \label{RWTS-etaandTDef} 
    \left. \displaystyle\frac{\partial}{\partial t}  \right|_M = T + \eta,
\end{equation}
where  $ T $  is a tangent vector field and  $\eta$  is normal to  $M$.

\begin{rem} \textbf{Assumptions}. Because of Lemma \ref{LemmaLn1f0CurvTensor}, the Robertson-Walker space   $L^n_1(f,c)$  has constant curvature if and only if the equation
 $\frac{f''}{f} = \frac{f'{}^2 + c}{f^2} $
is satisfied, \cite{ChenJVanderVeken2007}. In this case, every PMCV submanifold is trivially biconservative (see Proposition \ref{PROPBiconser}). On the other hand, if $T = 0 $  on a non-empty open subset $\mathcal{O} $  of $M$, then any component $  \mathcal{O}_1 $  of $  \mathcal{O} $  is a horizontal slice, i.e., $  (\mathcal{O}_1, g) \subset \{ t_0 \} \times_{f(t_0)} {R}^3(c) $. Hence, throughout this paper we have the following assumptions:
\begin{itemize}
\item $\displaystyle \frac{f''(t)}{f(t)} - \frac{f'(t)^2 + c}{f(t)^2} \neq 0$  for any $t \in I $,
\item $T$  does not vanish on $M$,
\item $\{e_1, e_2\} $  is an orthonormal basis for the tangent bundle of $  M $, where $e_1$  is proportional to $T$,
\item All surfaces are connected and all vector fields are smooth.
\end{itemize}
\end{rem}

We shall denote the coefficient of the second fundamental form by $h^\alpha_{ij}$, i.e., 
$$h^\alpha_{ij}=\langle h(e_i,e_j),e_\alpha\rangle=\langle A_{e_\alpha}e_i,e_j\rangle,\qquad i,j=1,2,$$
where $e_\alpha$ is  a unit normal vector field.

\subsection{Surfaces in a Lorentzian Product Space}
Now, consider the case $f=1$ and $c=\pm1$, i.e.,  $M$ is a non-degenerate surface in the Cartesian product space $\mathbb E^1_1\times R^{n-1}(c)$. Let $\hat\nabla$ denote the Levi-Civita connection of the flat ambient space $\mathbb E^{n+1}_r$ with the metric tensor
$$\tilde g=-dx_1^2+cdx_2^2+\sum\limits_{i=3}^{n+1}dx_i^2.$$
In this case, $\hat\nabla$ and $\widetilde\nabla$  are related with
\begin{subequations}\label{LCConnectionsRelatedAll}
\begin{eqnarray}
\label{LCConnectionsRelated1}\hat\nabla_X Y&=&\widetilde\nabla_X Y-c\left(\langle X,Y\rangle+\langle X,T\rangle\langle Y,T\rangle\right) e_{n+1}\\
\label{LCConnectionsRelated2}\hat\nabla_X \xi&=&\widetilde\nabla_X \xi-c\langle X,\eta\rangle\langle \xi,\eta\rangle e_{n+1}
\end{eqnarray}
whenever $X,Y$ are tangent to $M$ and $\xi$ is a normal vector field, where $e_{n+1}$ denotes the unit normal vector of $\mathbb E^1_1\times R^{n-1}(c)$ in  $\mathbb E^{n+1}_r$. Furthermore, we have
\begin{equation}\label{LCConnectionsRelated3}
\hat\nabla_X e_{n+1}=(X+\langle X,T\rangle T)+\langle X,T\rangle \eta.
\end{equation}
\end{subequations}


\section{Biconservative Surfaces in $L^n_1(f,c)$} 

In this section, we are going to consider space-like PMCV surfaces in the Robertson-Walker space $L^n_1(f,c)$, where $n\geq 4$.
\subsection{Space-like surfaces}
First, let $M$ be an oriented space-like surface in $L^n_1(f,c)$. In this case, by considering the decomposition \eqref{RWTS-etaandTDef}, we define the unit normal vector field $e_3$ and an `angle'  function $\theta$ by
\begin{equation}\label{RWTS-etaandTDefNew}
\left.\displaystyle\frac{\partial}{\partial t} \right|_M=\sinh\theta \, e_1+\cosh\theta \, e_3.  
\end{equation}
Then, by considering Lemma \ref{LemmaLn1f0CurvTensor}, one can observe that the Codazzi equation \eqref{CodEqMostGeneral} becomes
\begin{eqnarray}    
    \label{CodazziLn1fc1} 0 &=& \left(\nabla^\perp_{e_1} h\right)(e_2, e_1) - \left(\nabla^\perp_{e_2} h\right)(e_1, e_1), \\
    \label{CodazziLn1fc2}\sinh\theta\left(-\frac{f''}{f} + \frac{f'{}^2 + c}{f^2}\right)\eta &=& \left(\nabla^\perp_{e_1} h\right)(e_2, e_2) - \left(\nabla^\perp_{e_2} h\right)(e_1, e_2).
\end{eqnarray}

We are going to use the following lemma.
\begin{lemma}
Let $M$ be a space-like surface in $L^n_1(f,c)$, where $n\geq 4$. Then, the equations 
\begin{subequations}\label{SpaceLikeTNablaXpdte1ALL}
\begin{eqnarray}
\label{SpaceLikeTNablaXpdte1T} e_1(\theta)\cosh\theta \, e_1+\sinh\theta \nabla_{e_1}e_1-\cosh\theta A_{e_3}e_1&=&\frac{f'}{f}\cosh^2\theta e_1,\\
\label{SpaceLikeTNablaXpdte2T} e_2(\theta )\cosh\theta \, e_1+\sinh\theta \nabla_{e_2}e_1-\cosh\theta A_{e_3}e_2&=&\frac {f'}{f}e_2,\\
\label{SpaceLikeTNablaXpdte1N}e_1(\theta )\sinh\theta \, e_3+\sinh\theta h(e_1,e_1)+\cosh\theta\nabla^\perp_{e_1}e_3&=&\frac{f'}{f}\cosh\theta\sinh\theta e_3,\\
\label{SpaceLikeTNablaXpdte2N}e_2(\theta )\sinh\theta \, e_3+\sinh\theta h(e_1,e_2)+\cosh\theta\nabla^\perp_{e_2}e_3&=&0,
\end{eqnarray}
\end{subequations}
are satisfied, where $e_3$ is the unit normal vector field defined by \eqref{RWTS-etaandTDefNew}.
\end{lemma}

\begin{proof}
Consider a space-like surface in $L^n_1(f,c)$ and let $e_3$ be the vector field defined by \eqref{RWTS-etaandTDefNew}.
Note that Lemma \ref{LemmaLn1f0LCConnect}  implies 
\begin{eqnarray}
\label{LemmaL41f0SPLEqProof3} \widetilde\nabla_{e_1}\displaystyle\frac{\partial}{\partial t} &=&\frac{f'}{f}(\cosh^2\theta e_1+\cosh\theta\sinh\theta e_3),\\
\label{LemmaL41f0SPLEqProof7} \widetilde\nabla_{e_2}\displaystyle\frac{\partial}{\partial t} &=&\frac {f'}{f}e_2
\end{eqnarray}
and from \eqref{RWTS-etaandTDefNew}, we have
\begin{align}\label{SpaceLikeTNablaXpdtGeneral}
\begin{split}
\widetilde\nabla_{X}\displaystyle\frac{\partial}{\partial t} =&X(\theta )(\cosh\theta \, e_1+\sinh\theta \, e_3)+\sinh\theta\left(\nabla_{X}e_1+h(e_1,X)\right)\\&+\cosh\theta\left(-A_{e_3}(X)+\nabla^\perp_{X}e_3\right)
\end{split}
\end{align}
whenever $X$ is tangent to $M$. By combining \eqref{SpaceLikeTNablaXpdtGeneral} with \eqref{LemmaL41f0SPLEqProof3} and \eqref{LemmaL41f0SPLEqProof7}, we get
\begin{align}\label{LemmaL41f0SPLEqProof3a}
\begin{split}
\frac{f'}{f}(\cosh^2\theta e_1+\cosh\theta\sinh\theta e_3)=&\Big(e_1(\theta )\cosh\theta \, e_1+\sinh\theta \nabla_{e_1}e_1-\cosh\theta A_{e_3}(e_1)\Big)\\
&+\Big(e_1(\theta )\sinh\theta \, e_3+\cosh\theta\nabla^\perp_{e_1}e_3+\sinh\theta h(e_1,e_1)\Big)
\end{split}
\end{align}
and
\begin{align}\label{LemmaL41f0SPLEqProof7a}
\begin{split}
\frac {f'}{f}e_2=&\Big(e_2(\theta )\cosh\theta \, e_1+\sinh\theta \nabla_{e_2}e_1-\cosh\theta A_{e_3}(e_2)\Big)\\
&+\Big(e_2(\theta )\sinh\theta \, e_3+\cosh\theta\nabla^\perp_{e_2}e_3+\sinh\theta h(e_1,e_2)\Big),
\end{split}
\end{align}
respectively. The tangential and normal parts of \eqref{LemmaL41f0SPLEqProof3a} and \eqref{LemmaL41f0SPLEqProof7a} give \eqref{SpaceLikeTNablaXpdte1ALL}.
\end{proof}


We also need the following lemma.
\begin{lemma} \label{LemmaSpacelikeTRRH}
Let $M$ be a space-like surface in the Robertson-Walker space $L^n_1(f,c)$. Then, the mean curvature vector $H$ of $M$ satisfies
\begin{equation}\label{SpacelikeTRRH}
\tr (\widetilde{R}(\cdot,H)\cdot)^T=\left(\frac{f''}f - \frac{f'{}^2+c}{f^2}\right)\langle H,\eta\rangle T,
\end{equation}
where $T$ and $\eta$ are the vector fields defined by \eqref{RWTS-etaandTDef}.
\end{lemma}

\begin{proof}
We define a function $H_0$ and smooth vector fields $\bar{e_1}$ and $\bar{H}$ by the orthogonal decomposition 
\begin{align}\label{SpacelikeEq1}
\begin{array}{rclcl}
e_1&=&-\sinh\theta\displaystyle\frac{\partial}{\partial t} &+&\bar{e_1},\\
H&=&-H_0\displaystyle\frac{\partial}{\partial t} &+&\bar{H},
\end{array}
\end{align}
from which we obtain
\begin{align}\label{SpacelikeEq2}
\begin{split}
\langle\bar{e_1},\bar{e_1}\rangle=&\cosh^2\theta,\\
\langle\bar{e_1},\bar{H}\rangle=&H_0\sinh\theta,
\end{split}
\end{align} 
where we have $H_0=\langle H,\eta\rangle$. 

On the other hand, by considering Lemma \ref{LemmaLn1f0CurvTensor}, we obtain
\begin{eqnarray*}
\widetilde{R}(e_1,H)e_1&=&\sinh^2\theta \widetilde{R}(\displaystyle\frac{\partial}{\partial t} ,\bar H)\displaystyle\frac{\partial}{\partial t} +H_0\sinh\theta \widetilde{R}(\bar{e_1},\displaystyle\frac{\partial}{\partial t} )\displaystyle\frac{\partial}{\partial t} -\sinh\theta\widetilde{R}(\displaystyle\frac{\partial}{\partial t} ,\bar H)\bar{e_1}\\
&&-H_0\widetilde{R}(\bar{e_1},\displaystyle\frac{\partial}{\partial t} )\bar{e_1}+\widetilde{R}(\bar{e_1},\bar H)\bar{e_1}\\
&=&\frac{f''}f\left(\sinh^2\theta \bar H-H_0\sinh\theta\bar {e_1}-\sinh\theta\langle \bar H,\bar {e_1}\rangle\displaystyle\frac{\partial}{\partial t} +H_0\langle \bar {e_1},\bar {e_1}\rangle\displaystyle\frac{\partial}{\partial t} \right).\\
&&+\frac{f'{}^2+c}{f^2}\left(\langle \bar H,\bar {e_1}\rangle \bar {e_1}-\langle \bar {e_1},\bar {e_1}\rangle \bar H\right)
\end{eqnarray*}
By combining this equation with \eqref{SpacelikeEq2}, we get
\begin{align}\label{SpacelikeEq3}
\begin{split}
\widetilde{R}(e_1,H)e_1=&\frac{f''}f\left(\sinh^2\theta \bar H-H_0\sinh\theta\bar {e_1}-H_0\sinh^2\theta\displaystyle\frac{\partial}{\partial t} +H_0\cosh^2\theta\displaystyle\frac{\partial}{\partial t} \right).\\
&+\frac{f'{}^2+c}{f^2}\left(H_0\sinh\theta \bar {e_1}-\cosh^2\theta \bar H\right).
\end{split}
\end{align}

Note that \eqref{SpacelikeEq1} and  \eqref{SpacelikeEq2} imply
\begin{align}\label{SpacelikeEq4}
\begin{split}
(\bar H)^T=&H_0 \sinh\theta e_1,\\
(\bar e_1)^T=& \cosh^2\theta e_1.
\end{split}
\end{align}
By combining \eqref{SpacelikeEq3} with \eqref{SpacelikeEq4}, we obtain
\begin{equation}\label{SpacelikeRe1He1}
\left(\widetilde{R}(e_1,H)e_1\right)^T=0.
\end{equation}

Similarly, we get
\begin{eqnarray*}
\widetilde{R}(e_2,H)e_2&=&\frac{f''}f H_0\displaystyle\frac{\partial}{\partial t} -\frac{f'{}^2+c}{f^2}  \bar H
\end{eqnarray*}
from which, along with \eqref{SpacelikeEq4}, we obtain
\begin{align}\label{SpacelikeRe2He2}
\begin{split}
\left(\widetilde{R}(e_2,H)e_2\right)^T =&\left(\frac{f''}f-\frac{f'{}^2+c}{f^2}\right) H_0\sinh\theta e_1.
\end{split}
\end{align}

Finally, we combine \eqref{SpacelikeRe1He1} and  \eqref{SpacelikeRe2He2} to get \eqref{SpacelikeTRRH}.
\end{proof}


\subsection{PMCV Surfaces in $L^n_1(f,c)$}
Let $M$ be a space-like PMCV surface in $L^n_1(f,c)$. By considering Proposition \ref{PROPBiconser} one can get that  $M$ is biconservative if and only if $$\tr (\widetilde{R}(\cdot,H)\cdot)^T=0$$ 
which is equivalent to
\begin{equation}\label{PropSpacelikeBicProofEq1}
\langle H,\eta\rangle=0
\end{equation}
because of Lemma \ref{LemmaSpacelikeTRRH}. Since $\eta$ is time-like, \eqref{PropSpacelikeBicProofEq1} implies that $H$ is space-like at every point. Therefore, we have the following corollary.
\begin{corry}
There are no marginally trapped biconservative PMCV surface in $L^n_1(f,c)$.
\end{corry}

\begin{proposition}\label{PropNonDegeBicCond}
Let $M$ be a space-like surface in $L^n_1(f,c)$ and assume that $\eta\neq0$. Then, $M$ is a biconservative PMCV surface if and only if there exists  a non-zero constant $H_0$ and a unit normal vector field $e_4$ such that
\begin{eqnarray}
\label{PropSpacelikeBicCondEq1}\nabla^\perp e_4&=&0,\qquad \langle e_4,\eta\rangle=0,\\ 
\label{PropSpacelikeBicCondEq2}A_{e_4}&=&\left(\begin{array}{cc}0&0\\0&2H_0 \end{array}\right),\\
\label{PropSpacelikeBicCondEq3}A_\xi&=&\left(\begin{array}{cc}\gamma_\xi&0\\0&-\gamma_\xi\end{array}\right)\qquad \mbox{whenever $\langle e_4,\xi\rangle=0$},
\end{eqnarray}
where $\gamma_\xi\in C^\infty (M)$.
\end{proposition}

\begin{proof}
In order to prove the necessary condition, we assume that $M$ is biconservative and PMCV. Considering vector fields $e_1,e_3$ defined by \eqref{RWTS-etaandTDefNew}. Then, there exists a non-zero constant $H_0$ and a unit normal vector field $e_4$ such that $H=H_0e_4$ and \eqref{PropSpacelikeBicCondEq1}. Let $\xi$ be a normal vector field such that $\widetilde g(\xi,e_4)=0$  from which we have
\begin{equation}\label{PropSpacelikeBicProofEq2}
\tr A_{\xi}=0,\qquad \tr A_{e_4}=2H_0.
\end{equation}

Note that because of $\nabla^\perp e_4=0$, we have $$\langle\nabla^\perp_X e_3,e_4\rangle=0$$ whenever $X$ is a tangent vector field. By combining this equation with  \eqref{SpaceLikeTNablaXpdte1N} and \eqref{SpaceLikeTNablaXpdte2N} and taking into account \eqref{PropSpacelikeBicProofEq2}, we obtain $$h^4_{11}=h^4_{12}=0, \qquad h^4_{22}=2H_0.$$ 
Therefore, we have \eqref{PropSpacelikeBicCondEq2}.
Moreover, the Ricci equation \eqref{RicciLn1fc} implies
$$0=R^D(e_1,e_2)e_4=h(e_1,A_{e_4}e_2)-h(A_{e_4}e_1,e_2)=2H_0h(e_1,e_2).$$
By combining this equation with the first equation  in \eqref{PropSpacelikeBicProofEq2}, we get \eqref{PropSpacelikeBicCondEq3}.

Conversely, if \eqref{PropSpacelikeBicCondEq2} and \eqref{PropSpacelikeBicCondEq3} are satisfied, then we have $H=H_0 e_4$. Furthermore, \eqref{PropSpacelikeBicCondEq1} implies $\nabla^\perp H=0$ and \eqref{PropSpacelikeBicProofEq1}. This completes the proof.
\end{proof}


Next, by using this proposition, we obtain the following properties of biconservative PMCV surfaces.
\begin{lemma}\label{CorryofPropNonDegeBicCond}
Let $M$ be a space-like biconservative PMCV surface in $L^n_1(f,c)$ and $p\in M$. Then the vector fields $e_1,\ e_2$ and $e_3$ defined by \eqref{RWTS-etaandTDefNew} satisfy
\begin{subequations}\label{CorryofPropNonDegeBicCondAll}
\begin{eqnarray}
\label{CorryofPropNonDegeBicCond0} h(e_1,e_2)&=&0,\\
\label{CorryofPropNonDegeBicCond1} \nabla_{e_1}e_1=\nabla_{e_1}e_2=\nabla_{e_2}e_1=\nabla_{e_2}e_2&=&0,\\
\label{LemmaL41f0SPLEq5} \frac{f'}{f}=\cosh\theta\gamma_{e_3}  ,\qquad&& e_2(\theta)=0.
\end{eqnarray}
Consequently, there exists a local coordinate system $(\mathcal N_p,(u,v))$ such that $\mathcal N_p\ni p$ and 
\begin{eqnarray}
\label{CorryofPropNonDegeBicCond3}  \left.e_1\right|_{\mathcal N_p}=-\sinh\theta\partial_u, \qquad \left.e_2\right|_{\mathcal N_p}=\partial_v.
\end{eqnarray}
\end{subequations}

\end{lemma}

\begin{proof}
Let $M$ be biconservative and PMCV. Note that \eqref{CorryofPropNonDegeBicCond0} directly follows from \eqref{PropSpacelikeBicCondEq2} and \eqref{PropSpacelikeBicCondEq3}. 

From \eqref{PropSpacelikeBicProofEq1} and  \eqref{PropSpacelikeBicCondEq3} we have 
\begin{equation}\label{CorryofPropNonDegeBicCondPrEq0}
A_{e_3}e_1=\gamma_{e_3}e_1,\qquad A_{e_3}e_2=-\gamma_{e_3}e_2.
\end{equation} 
The first equation in \eqref{CorryofPropNonDegeBicCondPrEq0} and \eqref{SpaceLikeTNablaXpdte1T}  imply $\nabla_{e_1}e_1=0$. Moreover, the Codazzi equation \eqref{CodazziLn1fc2} and \eqref{CorryofPropNonDegeBicCond0} imply
\begin{equation}\label{CorryofPropNonDegeBicCondPrEq1}
\sinh\theta\cosh\theta\left(-\frac {f''}f+\frac{f'{}^2+c}{f^2}\right)e_3=\nabla^\perp_{e_1}h(e_2,e_2)+h(\nabla_{e_2}e_1,e_2)+ h(e_1,\nabla_{e_2}e_2).
\end{equation}
By taking the inner product of both sides of \eqref{CorryofPropNonDegeBicCondPrEq1} with $e_4$ and using \eqref{PropSpacelikeBicCondEq1} and \eqref{PropSpacelikeBicCondEq2}, we obtain
$$
H_0 \langle \nabla_{e_2} e_1, e_2 \rangle= 0.
$$
Thus, we have \eqref{CorryofPropNonDegeBicCond1}. From \eqref{SpaceLikeTNablaXpdte2T} and  \eqref{CorryofPropNonDegeBicCond1} we get the first equation in \eqref{LemmaL41f0SPLEq5}.

Finally, by considering \eqref{CorryofPropNonDegeBicCond1} and \eqref{LemmaL41f0SPLEq5}, we obtain
$$[-\frac 1{\sinh\theta} e_1, e_2]=0.$$
Therefore, there exists a local coordinate system $(u,v)$ such that $-\frac 1{\sinh\theta} e_1 = \partial_u$ and $e_2 = \partial_v$, which yields \eqref{CorryofPropNonDegeBicCond3}.
\end{proof}

On the other hand, as a result of Ricci equation \eqref{RicciLn1fc}, a surface has flat normal bundle if all of its shape operators can be diagonalized simultaneously (See \cite[Proposition 3.1]{ChenJVanderVeken2007}). Hence, we have the following result.
\begin{corry}\label{CorryFlatNB}
If $M$ is a biconservative PMCV surface in $L^n_1(f,c)$, then it has flat normal bundle.
\end{corry}


\section{PMCV Surfaces in $L^4_1(f,0)$} 
In this section, we are going to consider space-like PMCV surfaces in the Robertson-Walker spacetime $L^4_1(f,0)$.

First, we obtain the following lemma by using Proposition \ref{PropNonDegeBicCond}.
\begin{lemma}\label{LemmaL41f0SPL}
Let $M$  be a space-like surface in $L^4_1(f,c)$ and $\{e_1,e_2;e_3,e_4\}$ be the orthonormal frame field defined by \eqref{RWTS-etaandTDefNew}. If $M$ is  PMCV and biconservative, then
\begin{subequations}\label{LemmaL41f0SPLALLEqs}
\begin{eqnarray}
\label{LemmaL41f0SPLEq1}\widetilde\nabla_{e_1}e_1=-\gamma e_3 ,&\qquad& \widetilde\nabla_{e_2}e_1=0,\\
\label{LemmaL41f0SPLEq2}\widetilde\nabla_{e_1}e_2=0 ,&\qquad& \widetilde\nabla_{e_2}e_2=\gamma e_3+2H_0e_4,\\
\label{LemmaL41f0SPLEq3}\widetilde\nabla_{e_1}e_3=-\gamma e_1 ,&\qquad& \widetilde\nabla_{e_2}e_3=\gamma e_2,\\
\label{LemmaL41f0SPLEq4}\widetilde\nabla_{e_1}e_4=0 ,&\qquad& \widetilde\nabla_{e_2}e_4=-2H_0e_2.
\end{eqnarray}
\end{subequations}
\end{lemma}

\begin{proof}
Let $M$ be  PMCV and biconservative. Then, Proposition \ref{PropNonDegeBicCond} implies \eqref{PropSpacelikeBicCondEq1}, \eqref{PropSpacelikeBicCondEq2}, \eqref{PropSpacelikeBicCondEq3} and we have \eqref{CorryofPropNonDegeBicCond1}  because of Lemma \ref{CorryofPropNonDegeBicCond}. By combining \eqref{PropSpacelikeBicCondEq1} and \eqref{PropSpacelikeBicCondEq2}, we obtain \eqref{LemmaL41f0SPLEq4}. Since the co-dimension of $M$ is 2, \eqref{PropSpacelikeBicCondEq1} also implies $\nabla^\perp e_3=0$.

On the other hand, from \eqref{PropSpacelikeBicCondEq3} for $\xi=e_3$ we get
\begin{equation}\label{LemmaL41f0SPLEqProof1}
A_{e_3}=\left(\begin{array}{cc}\gamma&0\\0&-\gamma\end{array}\right),
\end{equation}
where we put $\gamma_{e_3}=\gamma$.   
Therefore, from \eqref{LemmaL41f0SPLEqProof1} we obtain \eqref{LemmaL41f0SPLEq3}. Note that by using \eqref{AhRelatedBy}, \eqref{PropSpacelikeBicCondEq3} and \eqref{LemmaL41f0SPLEqProof1},  we obtain the second fundemental form $h$ by 
\begin{equation}\label{LemmaL41f0SPLEqProof1b}
h(e_1,e_1)=-\gamma e_3,\qquad h({e_2},e_2)=\gamma e_3+2H_0e_4.
\end{equation}
By combining \eqref{CorryofPropNonDegeBicCond0}, \eqref{CorryofPropNonDegeBicCond1} and  \eqref{LemmaL41f0SPLEqProof1b}, we get \eqref{LemmaL41f0SPLEq1} and \eqref{LemmaL41f0SPLEq2}. 
\end{proof}


Next,  we obtain a local parametrization for space-like biconservative surfaces. 

\begin{lemma}\label{Lemma2L41f0SPL}
Let $M$  be a space-like surface in $L^4_1(f,0)$, $p\in M$ and $\{e_1,e_2;e_3,e_4\}$ be the orthonormal frame field defined by \eqref{RWTS-etaandTDefNew}. If $M$ is PMCV and biconservative, then $p$ has a neighborhood $\mathcal N_p$ parametrized by 
\begin{equation}
\label{Lemma2L41f0SEq1}\phi(u,v)=\left(u,\frac1{af(u)}\sin a v,\frac1{af(u)}\cos a v,y(u)\right)
\end{equation}
and we have
\begin{equation}
\label{Lemma2L41f0SEq3}e_4=\frac 1f\left(0,-\frac{2 H_0 }{a}\sin a v,-\frac{2 H_0 }{a}\cos a v,c_2\right),
\end{equation}
for some constants $a,c_2$ satisfying 
\begin{equation}
\label{Lemma2L41f0LEq4}4H_0^2+c_2^2a^2=a^2, \qquad c_2>0,
\end{equation}
where $H_0$ is the mean curvature of $M$ and $\phi$ is the position vector of $M$.
\end{lemma}

\begin{proof}
Let $M$  be  space-like  PMCV and biconservative. Then, because of Lemma \ref{LemmaL41f0SPL}, the equations appearing in \eqref{LemmaL41f0SPLALLEqs} is satisfied. Consider a local coordinate system $(\mathcal N_p,(u,v))$ described in Lemma \ref{CorryofPropNonDegeBicCond} near to $p$ and let
\begin{equation}\label{Lemma2L41f0SPLEq1}
\phi(u,v)=(\mathcal T(u,v),\widetilde\phi(u,v))
\end{equation}
be the position vector of $\mathcal N_p$. Because of the definitions of $e_1$ and $e_2$, from \eqref{CorryofPropNonDegeBicCond3} we have $\frac{\partial}{\partial u}\mathcal T=1,\ \frac{\partial}{\partial v}\mathcal T=0$ which imply
\begin{equation} \label{Lemma2L41f0SPLEq2}
\mathcal T(u,v)=u.
\end{equation}
Note that the second equation in \eqref{LemmaL41f0SPLEq5} implies $\theta=\theta(u)$ on $\mathcal N_p$.

By considering the first equation in \eqref{LemmaL41f0SPLEq2}, we obtain
$$0=\widetilde\nabla_{\partial_v}\partial_u=(0,\widetilde\phi_{uv})+\frac{f'}{f}(0,\widetilde\phi_{v})$$
from which we obtain
\begin{equation} \label{Lemma2L41f0SPLEq3}
e_2=\widetilde\phi_v=\frac 1f(0,\alpha),\qquad g_0(\alpha,\alpha)=1.
\end{equation}
where  $\alpha$ is a smooth $\mathbb R^3$ valued function. Next, by considering Lemma \ref{LemmaLn1f0LCConnect}, we use \eqref{Lemma2L41f0SPLEq1}-\eqref{Lemma2L41f0SPLEq3} to get
\begin{equation} \label{Lemma2L41f0SPLEq4}
\widetilde\nabla_{e_2}\widetilde\nabla_{e_2}e_2=\frac 1f\left(0,\alpha''\right)+\frac{f'^2}{f^2}e_2.
\end{equation}
On the other hand, by a direct computation using \eqref{LemmaL41f0SPLEq2}, \eqref{LemmaL41f0SPLEq3} and \eqref{LemmaL41f0SPLEq4} we get
\begin{equation} \label{Lemma2L41f0SPLEq5}
\widetilde\nabla_{e_2}\widetilde\nabla_{e_2}e_2=\left( \gamma^2-4H_0^2\right)e_2.
\end{equation}
By combining \eqref{Lemma2L41f0SPLEq3}-\eqref{Lemma2L41f0SPLEq5} we get
$$\alpha''(v)=\left(-\gamma(u)^2\sinh^2\theta(u)-4H_0^2\right)\alpha(v).$$
Consequently, we have $\gamma(u)^2\sinh^2\theta(u)+4H_0^2=a^2$ for a constant $a>0$ and $\alpha''+a^2\alpha=0$ from which we have
\begin{equation} \label{Lemma2L41f0SPLEq6}
\alpha(v)=(\cos av,-\sin av,0).
\end{equation}
Finally, by a direct computation using \eqref{Lemma2L41f0SPLEq1}-\eqref{Lemma2L41f0SPLEq3} and \eqref{Lemma2L41f0SPLEq6} we obtain  \eqref{Lemma2L41f0SEq1}.

Next, we consider the first equation in \eqref{LemmaL41f0SPLEq4} to get
$$0=\widetilde\nabla_{\partial_u}e_4=\frac{\partial}{\partial u}e_4+\frac{f'}{f}e_4$$
which implies
$$e_4=\frac 1f\left(0,N_1(v),N_2(v),N_3(v)\right)$$
for some smooth functions $N_1,N_2,N_3$. Furthermore, the second equation of \eqref{LemmaL41f0SPLEq4} implies
$$-2H_0e_2=\widetilde\nabla_{e_2}e_4=\frac 1f\left(0,N_1'(v),N_2'(v),N_3'(v)\right)$$
from which we obtain \eqref{Lemma2L41f0SEq3} because of \eqref{Lemma2L41f0SPLEq3} and $\langle e_2,e_4\rangle=0$.
\end{proof}


Now, we are ready to prove the main result of this section.

\begin{theorem}\label{THML41f0SPL}
The Robertson-Walker spacetime $L^4_1(f,0)$ admits a space-like, biconservative PMCV surface $M$ with mean curvature $H_0$ if and only if $f$ satisfies
\begin{equation}\label{THML41f0SPLEq1}
\left(a^2-4 H_0^2\right) f^3 f''-\left(f'{}^2-\left(a^2-4 H_0^2\right) f^2\right)^2-f'{}^4=0
\end{equation}
for a constant $a$ such that $a^2-4 H_0^2>0$. In this case, $M$ is locally congruent to the rotational surface 
\begin{equation}\label{THML41f0SPLEq2}
\phi(u,v)=\left(u,\frac{1}{a f(u)}\sin a v,\frac{1}{a f(u)}\cos a v,\frac{2 H_0}{a^2 c_2 f(u)}\right),
\end{equation}
where  $c_2$ is a constant satisfying \eqref{Lemma2L41f0LEq4}.
\end{theorem}

\begin{proof}
First, we are going to prove the necessary condition. Assume the existence of a space-like biconservative PMCV surface  $M$ in $L^4_1(f,0)$. Then  \eqref{PropSpacelikeBicProofEq1} is satisfied and we are going to consider a local coordinate system $(u,v)$ 
described in  Lemma \ref{Lemma2L41f0SPL}. Note that \eqref{Lemma2L41f0SEq1} and \eqref{Lemma2L41f0SEq3} implies
\begin{equation}\label{THML41f0SPLProofEq1}
a^2 c_2 f^2 y'+2 H_0 f'=0
\end{equation}
because $\langle e_1,e_4\rangle=0.$ By solving \eqref{THML41f0SPLProofEq1}, we  get
$$y(u)=\frac{2 H_0}{a^2 c_2 f(u)}+c_3$$
for a constant $c_3$ which can be assumed to be zero after a suitable translation. Therefore, we have \eqref{THML41f0SPLEq2}. 

By using  \eqref{CorryofPropNonDegeBicCond3} and \eqref{Lemma2L41f0SEq1} we get 
$$e_3=\frac{1}{\sqrt{f'{}^2-\left(a^2-4 H_0^2\right) f^2}}\left( {f'} ,- {a c_2^2 \sin  av } ,- {a c_2^2 \cos  av } ,- {2 c_2 H_0} \right)$$
from which we obtain
\begin{equation}\label{THML41f0SPLProofEq2}
\widetilde\nabla_{e_1}e_1=\frac{f'^4-\left(a^2-4 H_0^2\right) f^3 f''}{f \left(f'{}^2-\left(a^2-4 H_0^2\right) f^2\right){}^{3/2}}e_3.
\end{equation}
On the other hand, \eqref{CorryofPropNonDegeBicCond3} and \eqref{Lemma2L41f0LEq4} implies
\begin{equation}\label{THML41f0SPLProofEq3}
\widetilde\nabla_{e_2}e_2=\frac{\sqrt{f'{}^2-\left(a^2-4 H_0^2\right) f^2}}{f}e_3+2H_0e_4.
\end{equation}
By combining \eqref{THML41f0SPLProofEq2} and \eqref{THML41f0SPLProofEq3}, we get the mean curvature vector $H$ of $M$ by
\begin{equation*}
H=-\frac{\left(a^2-4 H_0^2\right) f^3 f''+2 \left(a^2-4 H_0^2\right) f^2 f'{}^2-\left(a^2-4 H_0^2\right){}^2 f^4-2 f'^4}{2f \left(f'{}^2-\left(a^2-4 H_0^2\right) f^2\right){}^{3/2}}e_3+H_0e_4.
\end{equation*}
Hence, \eqref{PropSpacelikeBicProofEq1} is equivalent to \eqref{THML41f0SPLEq1}.

The proof of the sufficient condition follows from a direct computation.
\end{proof}


\section{PMCV Surfaces in $L^5_1(f,0)$} 
In this section, we are going to consider biconservative PMCV surfaces in the Robertson-Walker space $L^5_1(f,0)$. 

Let $M$ be a space-like biconservative PMCV surface in $L^5_1(f,0)$. We choose the orthonormal frame field $\{e_1,e_2;e_3,e_4,e_5\}$, where $e_1,e_2,e_3$ are defined by \eqref{RWTS-etaandTDefNew} and $e_4$ is proportional to $H$, i.e., $H=H_0 e_4$. Then, since $e_4$ is parallel, we have
\begin{eqnarray}\label{SurfacesinL51fcEq0}
\nabla^\perp{e_4}=0,\qquad \langle\nabla_{e_1}^\perp{e_3},e_4\rangle=\langle\nabla_{e_1}^\perp{e_5},e_4\rangle=0.
\end{eqnarray}
 Proposition \ref{PropNonDegeBicCond} implies
\begin{eqnarray}\label{SurfacesinL51fcEq1}
A_{e_3}= \left(\begin{array}{cc}\gamma&0\\0&-\gamma\end{array}\right), \qquad A_{e_4}= \left(\begin{array}{cc}0&0\\0&2H_0 \end{array}\right), \qquad A_{e_5}= \left(\begin{array}{cc}\tau&0\\0&-\tau\end{array}\right),
\end{eqnarray}
where, for simplicity, we put $\gamma_{e_3}=\gamma$, $\gamma_{e_5}=\tau$ and $H_0$ denotes the mean curvature of $M$. Note that \eqref{AhRelatedBy} and \eqref{SurfacesinL51fcEq1}  imply
\begin{equation}\label{SurfacesinL51fcEq2}
h(e_1,e_1)=-\gamma e_3+\tau e_5,\qquad h({e_2},e_2)=\gamma e_3+2H_0e_4-\tau e_5
\end{equation}
and we have \eqref{CorryofPropNonDegeBicCond0} because of Lemma \ref{CorryofPropNonDegeBicCond}. By using \eqref{CorryofPropNonDegeBicCond0}, from \eqref{SpaceLikeTNablaXpdte2N} we get
\begin{eqnarray}\label{SurfacesinL51fcEq3}
\nabla_{e_2}^\perp{e_3}=\nabla_{e_2}^\perp{e_5}=0.
\end{eqnarray}
Moreover, \eqref{SpaceLikeTNablaXpdte1N} and \eqref{SurfacesinL51fcEq2} imply
\begin{equation}\label{SurfacesinL51fcEq4}
\nabla^\perp_{e_1}e_3=-\tanh\theta \tau e_5,\qquad \nabla^\perp_{e_1}e_5=-\tanh\theta \tau e_3,
\end{equation}
where the last equation follows from \eqref{SurfacesinL51fcEq0}. By summing up the equations \eqref{SurfacesinL51fcEq0}-\eqref{SurfacesinL51fcEq4}, we obtain the following lemma.
\begin{lemma}\label{LemmaL51fcSPL}
Let $M$  be a space-like surface in $L^5_1(f,0)$ with mean curvature $H_0$ and $\{e_1,e_2;e_3,e_4,e_5\}$ be the orthonormal frame field defined by \eqref{RWTS-etaandTDefNew} and  $H=H_0 e_4$. If $M$ is  PMCV and biconservative, then
\begin{subequations}\label{LemmaL51fcSPLALLEqs}
\begin{eqnarray}
\label{LemmaL51fcSPLEq1}\widetilde\nabla_{e_1}e_1=-\gamma e_3+\tau e_5 ,&\qquad& \widetilde\nabla_{e_2}e_1=0,\\
\label{LemmaL51fcSPLEq2}\widetilde\nabla_{e_1}e_2=0 ,&\qquad& \widetilde\nabla_{e_2}e_2=\gamma e_3+2H_0e_4-\tau e_5,\\
\label{LemmaL51fcSPLEq3}\widetilde\nabla_{e_1}e_3=-\gamma e_1-\tanh\theta \tau e_5,&\qquad& \widetilde\nabla_{e_2}e_3=\gamma e_2,\\
\label{LemmaL51fcSPLEq4}\widetilde\nabla_{e_1}e_4=0 ,&\qquad& \widetilde\nabla_{e_2}e_4=-2H_0e_2,\\
\label{LemmaL51fcSPLEq5}\widetilde\nabla_{e_1}e_5=-\tau e_1-\tanh\theta \tau e_3 ,&\qquad& \widetilde\nabla_{e_2}e_5=\tau e_2.
\end{eqnarray}
\end{subequations}
\end{lemma}


Next, by considering Lemma \ref{LemmaL51fcSPL}, we   obtain a local parametrization of $M$.
\begin{lemma}\label{Lemma2L51f0SPL}
Let $M$  be a space-like surface in $L^5_1(f,0)$, $p\in M$.  If $M$ is PMCV and biconservative, then $p$ has a neighborhood $\mathcal N_p$ congruent to the surface parametrized by 
\begin{equation}\label{Lemma2L51f0SPLEq1}
\phi(u,v)=\left(u,x(u)\sin a v,\frac1{af(u)}\cos a v,y(u),z(u)\right),
\end{equation}
 where $a>$ is a constant and $x$ is given by
\begin{equation}\label{Lemma2L51f0SPLEq2a}
x(u)=\frac1{af(u)}
\end{equation}
and $y$, $z$ are some smooth functions satisfying
\begin{equation}\label{Lemma2L51f0SPLEq2}
c_2 y+c_3z-\frac{2H_0}{a}x=0
\end{equation}
for some non-zero constants $c_2,c_3$ such that
\begin{equation}\label{Lemma2L51f0SPLEq3}
c_2^2+c_3^2+\frac{4H_0^2}{a^2}=1.
\end{equation}
\end{lemma}

\begin{proof}
Let $M$ be a space-like surface in $L^5_1(f,0)$. Assume that it is PMCV and biconservative and consider a local coordinate system $(\mathcal N_p,(u,v))$ described in Lemma \ref{CorryofPropNonDegeBicCond} near to $p$. By the same method used in the proof of Lemma \ref{Lemma2L41f0SPL}, we obtain
\begin{equation}\label{Lemma2L51f0SPLProofEq1}
\phi(u,v)=(u,\widetilde\phi(u,v)),
\end{equation}
for a smooth $\mathbb R^4$ valued function $\tilde\phi$, where $\phi$ is the position vector of $\mathcal N_p$.
Next, we use the second equations of \eqref{LemmaL51fcSPLEq2}-\eqref{LemmaL51fcSPLEq5} to get
\begin{equation} \label{Lemma2L51f0SPLProofEq2}
\widetilde\nabla_{e_2}\widetilde\nabla_{e_2}e_2=\left( \gamma^2-4H_0^2-\tau^2\right)e_2.
\end{equation}
By combining \eqref{LemmaL41f0SPLEq5} and \eqref{Lemma2L51f0SPLProofEq2} we get
$$\alpha''(v)+\left(\gamma(u)^2\sinh^2\theta(u)+\tau^2+4H_0^2\right)\alpha(v)=0.$$
Consequently, we have $\gamma(u)^2\sinh^2\theta(u)+\tau^2+4H_0^2=a^2$ for a constant $a>0$ and $\alpha''+a^2\alpha=0$ from which we have
\begin{equation} \label{Lemma2L51f0SPLProofEq3}
\alpha(v)=(\cos av,-\sin av,0,0).
\end{equation}
By using \eqref{Lemma2L41f0SPLEq3} and \eqref{Lemma2L51f0SPLProofEq3}, we obtain \eqref{Lemma2L51f0SPLEq1} and \eqref{Lemma2L51f0SPLEq2a} for some smooth functions $y,z$. Therefore, we have
\begin{equation} \label{Lemma2L51f0SPLProofEq4}
\phi_u=\left(1,-x'\sin a v,-x'\cos a v,y',z'\right).
\end{equation}

On the other hand, similar to proof of Lemma \ref{Lemma2L41f0SPL}, we observe that \eqref{LemmaL51fcSPLEq4} and \eqref{Lemma2L51f0SPLProofEq3} imply
\begin{align*} 
\begin{split}
\frac{\partial e_4}{\partial u}+\frac{f'}{f}e_4=0,\qquad& \frac{\partial e_4}{\partial v}=\frac {-2H_0}f(\cos av,-\sin av,0,0).
\end{split}
\end{align*}
By integrating these equations and using $\langle e_2,e_4\rangle=0$ we get
\begin{equation}\label{Lemma2L51f0SPLProofEq5}
e_4=\frac 1f\left(0,-\frac{2 H_0 }{a}\sin a v,-\frac{2 H_0 }{a}\cos a v,c_2,c_3\right),
\end{equation}
for some non-zero constants $c_2,c_3$ satisfying \eqref{Lemma2L51f0SPLEq3}. Finally, we consider \eqref{Lemma2L51f0SPLProofEq4} and \eqref{Lemma2L51f0SPLProofEq5} and $\langle e_1,e_4\rangle=0$ to get \eqref{Lemma2L51f0SPLEq2}.
\end{proof}


Now, we are ready to prove the main result of this section.
\begin{theorem}\label{THML51f0SPL}
Let $y,f$ be some functions satisfying the system given by
\begin{align}\label{THML51f0SPLEq1}
\begin{split}
-a^6 b^2 c_3^2 f^7 y' y''-a^4 c_3^2 c_4 f^3 f' f''-2 a^6 c_2 c_3^2 H_0 f^5 \left(f'' y'+f' y''\right)&\\
-2 a^2 c_4 f^2 f'{}^3 \left(a^2 c_3^2-8 c_2 H_0 f' y'\right)-4 a^4 b^2 f^6 f' y'{}^2 \left(a^2 c_3^2-4 c_2 H_0 f' y'\right)&\\
+a^2 f^4 f' \left(-12 a^4 c_2 c_3^2 H_0 f' y'+4 \left(a^4 c_3^2-12 a^2 \left(c_3^2-1\right) H_0^2-48 H_0^4\right) f'{}^2 y'{}^2\right)&\\
+2 a^4 b^4 f^8 f' y'{}^4+a^8c_3^4f^4 f' +2 c_4^2 f'{}^5&=0\\
a^4 c_3^2 f ^3 \left(f'  y'' -f''  y' \right)-a^2 b^4 f ^6 y' {}^3+a^2 b^2 f ^4 y'  \left(a^2 c_3^2-6 c_2 H_0 f'  y' \right)&\\
+f ^2 f'  \left(2 a^4 c_2 c_3^2 H_0+\left(a^4 c_3^2+12 a^2 \left(c_3^2-1\right) H_0^2+48 H_0^4\right) f'  y' \right)-2 c_2 c_4 H_0 f' {}^3&=0
\end{split}
\end{align}
for some non-zero constants $a,c_2,c_3$ satisfying $b^2=a^2-4 H_0^2>0$, where we put $c_4=a^2 c_3^2+4 H_0^2$. Then, the Robertson-Walker spaces $L^5_1(f,0)$ admits a space-like, biconservative PMCV surface $M$ with the mean curvature $H_0$ parametrized by 
\begin{equation}\label{THML51f0SPLEq2}
\phi(u,v)=\left(u,\frac{\sin a v}{a f(u)},\frac{\cos a v}{a f(u)},y(u),\frac{2 H_0-c_2 a^2 f(u)y(u)}{c_3a^2 f(u)}\right).
\end{equation}

Conversely, if a Robertson-Walker space $L^5_1(f,0)$ admits a space-like, biconservative PMCV surface, then $f$ must be a solution of \eqref{THML51f0SPLEq1} and the surface must be locally congruent to the surface given by \eqref{THML51f0SPLEq2}.
\end{theorem}
\begin{proof}
Let $M$ be a surface in given by \eqref{Lemma2L51f0SPLEq1}, \eqref{Lemma2L51f0SPLEq2a} and \eqref{Lemma2L51f0SPLEq2} for some non-zero constant $a,c_2,c_3,H_0$ satisfying \eqref{Lemma2L51f0SPLEq3}. Let $e_1,e_2,e_4$ be the vector fields given by \eqref{CorryofPropNonDegeBicCond3} and \eqref{Lemma2L51f0SPLProofEq5}. We also consider the unit normal vector fields $e_3$ and $e_5$ given by
\begin{equation}\label{Lemma2L51f0SPLAfterProofEq1}
e_3=\frac 1{V f  \sqrt{V^2 f ^2-1}}\frac 1f\left(V^2 f ^2, x'\sin a v , x'\cos a v ,y' ,z' \right),
\end{equation}
and
\begin{align}\label{Lemma2L51f0SPLAfterProofEq3}
\begin{split}
e_5=&\frac 1{V f}\left(0, \left(c_2 z' -c_3 y' \right)\sin  av , \left(c_2 z' -c_3 y' \right)\cos  av ,\frac{2 H_0}{a} z' +c_3 x' ,\right.\\&\left.
-\frac{2 H_0}{a} y' -c_2 x'  \right),
\end{split} 
\end{align}
where we put 
\begin{equation}\label{Lemma2L51f0SPLAfterProofEq2}
V=x'{}^2+y'{}^2+z'{}^2.
\end{equation} 
Then $\{e_1,e_2;e_3,e_4,e_5\}$ is an orthonormal frame field on $M$  defined by \eqref{RWTS-etaandTDefNew} and satisfying \eqref{LemmaL51fcSPLEq4}. Therefore, $M$ is a biconservative PMCV surface if and only if $H=2H_0e_4$. By using \eqref{Lemma2L51f0SPLAfterProofEq1} and \eqref{Lemma2L51f0SPLAfterProofEq3}, we get
\begin{align}\label{Lemma2L51f0SPLAfterProofEq4all}
\begin{split}
A_{e_3}(e_1)=&\frac{V \left(V^2 fc_3^2-2\right) f'c_3-V' fc_3}{\left(V^2 fc_3^2-1\right)^{3/2}}e_1,\\
A_{e_3}(e_2)=&\frac{a x'c_3+V^2 f'c_3}{V \sqrt{V^2 fc_3^2-1}},\\
A_{e_5}(e_1)=&\frac{f \left(a x' \left(c_2 z''-c_3 y''\right)-z' \left(a c_2 x''+2 H_0 y''\right)+y' \left(a c_3 x''+2 H_0 z''\right)\right)}{a V \left(V^2 f^2-1\right)}e_1,\\
A_{e_5}(e_2)=&\frac{a c_2 z'-a c_3 y'}{V}.
\end{split} 
\end{align}
By considering \eqref{Lemma2L51f0SPLAfterProofEq2} and \eqref{Lemma2L51f0SPLAfterProofEq4all}, we observe that the equation  $H=2H_0e_4$ is equivalent to the system
\begin{align}\label{Lemma2L51f0SPLAfterProofEq5}
\begin{split}
x' \left(a f^2 \left(y'{}^2+z'{}^2\right)-a-f x''\right)+f' \left(x'{}^2+y'{}^2+z'{}^2\right) \left(2 f^2 \left(x'{}^2+y'{}^2+z'{}^2\right)-3\right)&\\
+a f^2 x'^3-f \left(y' y''+z' z''\right)&=0\\
a^2 c_2 f^2 y'{}^2 z'a f x' \left(c_2 z''-c_3 y''\right)+z' \left(a c_2 \left(a f^2 x'{}^2-a-f x''\right)-2 H_0 f y''\right)&\\
+a^2 c_2 f^2 z'^3-a^2 c_3 f^2 y'^3++y' \left(a c_3 \left(f \left(x''-a f \left(x'{}^2+z'{}^2\right)\right)+a\right)+2 H_0 f z''\right)&=0.
\end{split} 
\end{align}
Finally, we use \eqref{Lemma2L51f0SPLEq2} in \eqref{Lemma2L51f0SPLAfterProofEq5} to get the system \eqref{THML51f0SPLEq1}.
\end{proof}


\section{PMCV Surfaces in Lorentzian Product Spaces} 
In this section, we are going to consider biconservative PMCV surfaces in the Cartesian product space $L^5_1(1,c)=\mathbb E^1_1\times R^4(c)\subset\mathbb E^6_r$, where $c=\pm1$.

Let $M$ be a space-like surface in $\mathbb E^1_1\times R^4(c)$ and assume that $M$ is PMCV and biconservative.  Consider the orthonormal frame field $\{e_1,e_2;e_3,e_4,e_5\}$ on $M$, where $e_1,e_2,e_3$ are defined by \eqref{RWTS-etaandTDefNew} and $e_4$ is proportional to $H$, i.e., $H=H_0 e_4$.  By using $f=1$, we consider \eqref{SpaceLikeTNablaXpdte2T} and \eqref{CorryofPropNonDegeBicCond1} to get $A_{e_3}e_2=0$. Therefore, \eqref{PropSpacelikeBicCondEq3} implies $A_{e_3}=0$. Therefore, we have 
\begin{equation}\label{LemmaL6RxS4SPreLemma1Eq1}
\gamma_{e_3}=0
\end{equation}
and from \eqref{SpaceLikeTNablaXpdte1T}  we obtain $e_1(\theta)=0$. Consequently, the second equation in \eqref{LemmaL41f0SPLEq5} implies 
\begin{equation}\label{LemmaL6RxS4SPreLemma1Eq2}
\theta=\theta_0
\end{equation}
for a constant $\theta_0\neq0$. Next, we consider \eqref{PropSpacelikeBicCondEq2}, \eqref{PropSpacelikeBicCondEq3}, \eqref{LemmaL6RxS4SPreLemma1Eq1} and \eqref{LemmaL6RxS4SPreLemma1Eq2} with Codazzi equations \eqref{CodazziLn1fc1} and \eqref{CodazziLn1fc2}, to get
\begin{equation}\label{LemmaL6RxS4SPreLemma1Eq3}
\left(\nabla^\perp_{e_2} \gamma_{e_5} e_5\right)=0
\end{equation}
and
\begin{equation}\label{LemmaL6RxS4SPreLemma1Eq4}
\sinh\theta_0\cosh\theta_0 ce_3=\nabla^\perp_{e_1}\left(-\gamma_{e_5} e_5\right),
\end{equation}
respectively. From \eqref{LemmaL6RxS4SPreLemma1Eq3} and \eqref{LemmaL6RxS4SPreLemma1Eq4}  we obtain
\begin{equation}\label{LemmaL6RxS4SPreLemma1Eq5}
\gamma_{e_5}=\tau_0
\end{equation}
for a non-zero constant $\tau_0$. 

\begin{rem}\label{NonExistenceRemark}
Before we proceed, we want to note that if $M$ is a biconservative PMCV surface in a totally geodesic hypersurface $\mathbb E^1_1\times R^3(c)$ of $\mathbb E^1_1\times R^4(c)$, the equation \eqref{LemmaL6RxS4SPreLemma1Eq4} turns into
$$\sinh\theta_0\cosh\theta_0 ce_3=0.$$
Hence, it turns out that there are no space-like biconservative PMCV surface in the 4-dimensional Cartesian product space $\mathbb E^1_1\times R^3(c)$ except the trivial case $T=0$ or, equivalently, $\theta_0=0$.
\end{rem}

Next we use \eqref{LCConnectionsRelatedAll}, \eqref{LemmaL6RxS4SPreLemma1Eq1}, \eqref{LemmaL6RxS4SPreLemma1Eq2} \eqref{LemmaL6RxS4SPreLemma1Eq5} in Lemma \ref{LemmaL51fcSPLALLEqs} for the case $f=1$ and $c=\pm1$ to get following result.
\begin{lemma}\label{LemmaRxRcSPLemma1}
Let $M$  be a space-like surface in $\mathbb E^1_1\times R^4(c)$ with mean curvature $H_0$ and $\{e_1,e_2;e_3,e_4,e_5\}$ be the orthonormal frame field defined by \eqref{RWTS-etaandTDefNew} and  $H=H_0 e_4$. If $M$ is  PMCV and biconservative, then the Levi-Civita connection $\hat\nabla$ of the flat ambient space $\mathbb E^6_r$ satisfies
\begin{subequations}\label{LemmaRxRcSPLqALL}
\begin{eqnarray}
\label{LemmaRxRcSPLq1}\hat\nabla_{e_1}e_1=\tau_0 e_5-c\cosh ^2\theta _0e_6 ,&\qquad& \hat\nabla_{e_2}e_1=0,\\
\label{LemmaRxRcSPLq2}\hat\nabla_{e_1}e_2=0 ,&\qquad& \hat\nabla_{e_2}e_2=2H_0e_4-\tau_0 e_5-ce_6,\\
\label{LemmaRxRcSPLq3}\hat\nabla_{e_1}e_3=-\tanh\theta_0 \tau_0 e_5+c\frac{\sinh2\theta_0}2 e_6,&\qquad& \hat\nabla_{e_2}e_3=0,\\
\label{LemmaRxRcSPLq4}\hat\nabla_{e_1}e_4=0 ,&\qquad& \hat\nabla_{e_2}e_4=-2H_0e_2,\\
\label{LemmaRxRcSPLq5}\hat\nabla_{e_1}e_5=-\tau_0 e_1-\tanh\theta_0 \tau_0 e_3 ,&\qquad& \hat\nabla_{e_2}e_5=\tau_0 e_2,\\
\label{LemmaRxRcSPLq6}\hat\nabla_{e_1}e_6=\cosh^2\theta_0 e_1+\frac{\sinh2\theta_0}2 e_3 ,&\qquad& \hat\nabla_{e_2}e_6=e_2
\end{eqnarray}
for some constants $\theta_0,\tau_0$, where $e_6$ denotes the unit normal vector field of  $\mathbb E^1_1\times R^4(c)$ in $\mathbb E^6_r$ and $H_0$ denotes the mean curvature of $M$ in $\mathbb E^1_1\times R^4(c)$. 
\end{subequations}
\end{lemma}


Next, we get the following non-existence result.
\begin{proposition}\label{NonExistenceE11H4}
There are no space-like biconservative PMCV surface in  $\mathbb E^1_1\times \mathbb H^4$.
\end{proposition}

\begin{proof}
Assume that $M$ is a space-like biconservative PMCV surface in  $\mathbb E^1_1\times \mathbb H^4$, i.e., $f=1$ and $c=-1$. Then,  by combining \eqref{LemmaRxRcSPLq2}, \eqref{LemmaRxRcSPLq4} with the Codazzi equation \eqref{CodazziLn1fc2}, we obtain 
\begin{equation}\label{NonExistenceE11H4E1}
-\sinh\theta_0\cosh\theta_0 e_3=-\tau_0\nabla^\perp_{e_1}e_5.
\end{equation}
However, the second equation in \eqref{LemmaRxRcSPLq5} and \eqref{NonExistenceE11H4E1} imply
$$
\sinh\theta_0\cosh\theta_0+\tau_0^2\tanh\theta_0=0
$$
which is a contradiction.
\end{proof}

\begin{rem}
The same result obtained for surfaces of the Riemannian product $\mathbb E^1\times \mathbb H^4$ in \cite[Lemma 3.1]{FetcuOniciucPinheiro2015}. 
\end{rem}


Now, we are ready to proved the main result of this section.
\begin{theorem}\label{THMRxRcSP1}
Let $M$ be a space-like surface in  $\mathbb E^1_1\times \mathbb S^4$. Then $M$ is biconservative and PMCV if and only if it is congruent to the surface locally parametrized by
\begin{align}\label{THMRxRcSP1Eq1}
\begin{split}
\phi(u,v)=&\left(-b_1u,\frac{\sqrt{b_1^2+1} \cos \left(\sqrt{b_1^2+2} u\right)}{\sqrt{b_1^2+2}},\frac{\sqrt{b_1^2+1} \sin \left(\sqrt{b_1^2+2} u\right)}{\sqrt{b_1^2+2}},b_2,\right.\\
&\left.b_3 \sin\frac{v}{b_3},b_3 \cos\frac{v}{b_3}\right)
\end{split}
\end{align}
for some non-zero constants $b_1,b_2,b_3$ satisfying $b_2^2+b_3^2=\frac{1}{b_1^2+2}$.
\end{theorem}
\begin{proof}
In order to prove the sufficient condition, assume that $M$ is biconservative PMCV and  let $p\in M$. Then, by Lemma \ref{LemmaRxRcSPLemma1} we have \eqref{LemmaRxRcSPLqALL}.  \eqref{LemmaRxRcSPLq1} and \eqref{LemmaRxRcSPLq2} implies $[e_1,e_2]=0$. So, there exists a local coordinate system $(\mathcal N_p,(u,v))$ such that $\mathcal N_p\ni p$ and 
\begin{equation}\label{THMRxRcSP1PrEq1}
e_1=\partial_u, \qquad e_2=\partial_v.
\end{equation}
This equation, the second equation in \eqref{LemmaRxRcSPLq1} and the first equation in \eqref{LemmaRxRcSPLq2} imply
\begin{equation}\label{THMRxRcSP1PrEq2}
e_1=\partial_u=(-\sinh\theta_0,\bar e_1)
\end{equation}
and
\begin{equation}\label{THMRxRcSP1PrEq3}
e_2=\partial_v=(0,\bar e_2)
\end{equation}
for some smooth $\mathbb R^5$-valued functions $\bar e_1=\bar e_1(u)$ and $\bar e_2=\bar e_2(v)$. 

On the other hand, by using the second equations in \eqref{LemmaRxRcSPLq2}, \eqref{LemmaRxRcSPLq4}-\eqref{LemmaRxRcSPLq6}, we get
$$
\hat\nabla_{e_2}\hat\nabla_{e_2}e_2=(0,\bar e_2'')=(4H_0^2+\tau_0^2+1)(0,\bar e_2).
$$
Since $\langle \bar e_2,\bar e_2\rangle=1$, up to a suitable rotation, we may assume
\begin{equation}\label{THMRxRcSP1PrEq5}
\bar e_2= \cos av C_1+\sin av C_2
\end{equation}
for some constant vectors $C_1,C_2\in\mathbb R^5$, where we put $a=\sqrt{4H_0^2+\tau_0^2+1}$. Since $\bar e_2$ is unit, $\{C_1,C_2\}$ is orthonormal.

Note that we have
\begin{equation}\label{THMRxRcSP1PrEq5e4}
e_4=(0,\bar e_4)
\end{equation}
for a $\mathbb R^5$-valued function $\bar e_4$. We use \eqref{LemmaRxRcSPLq4} and \eqref{THMRxRcSP1PrEq5e4} to obtain
\begin{align*} 
\begin{split}
\frac{\partial \bar e_4}{\partial u}=0,\qquad& \frac{\partial \bar e_4}{\partial v}=-2H_0(\cos av C_1+\sin av C_2)
\end{split}
\end{align*}
which implies
\begin{equation}\label{THMRxRcSP1PrEq6}
\bar e_4=\frac{-2H_0}{a}\sin av C_1+\frac{-2H_0}{a}\cos av C_2+C_3
\end{equation}
for a constant vector $C_3\in\mathbb R^5$. By considering $\langle \bar e_2,\bar e_4\rangle=0$ and $\langle \bar e_4,\bar e_4\rangle=1$, we obtain $\langle C_1,C_3\rangle=\langle C_2,C_3\rangle=0$ and $\langle C_3,C_3\rangle=\frac{\tau_0^2+1}{a^2}$. Up to a suitable rotation in $\mathbb R^5$, we may choose $C_1=(0,0,0,1,0),\ C_2=(0,0,0,0,-1),\ C_3=(0,0,\frac{\sqrt{\tau_0^2+1}}{a},0,0)$. Therefore, \eqref{THMRxRcSP1PrEq3}-\eqref{THMRxRcSP1PrEq6} imply
\begin{eqnarray}
\label{THMRxRcSP1PrEq7} e_2=\phi_v&=&(0,0,0,0,\cos av, -\sin av),\\
\label{THMRxRcSP1PrEq8} e_4&=&(0,0,0,\frac{\sqrt{\tau_0^2+1}}{a},\frac{-2H_0}{a}\sin av,\frac{2H_0}{a} \cos av).
\end{eqnarray}
Moreover, by a direct computation using \eqref{THMRxRcSP1PrEq2}, \eqref{THMRxRcSP1PrEq7} and $\langle e_1,e_2\rangle=\langle e_1,e_6\rangle=0$, we get
\begin{eqnarray}
\label{THMRxRcSP1PrEq9} e_1=\phi_u&=&(-\sinh\theta_0,\phi_1'(u),\phi_2'(u),\phi_3'(u),0, 0)
\end{eqnarray}
for some smooth functions $\phi_1,\phi_2,\phi_3$ satisfying
\begin{eqnarray}
\label{THMRxRcSP1PrEq10} \phi_1'{}^2+\phi_2'{}^2+\phi_3'{}^2=\cosh^2\theta_0.
\end{eqnarray}
Note that $\langle e_1,e_4\rangle=0$ implies
\begin{eqnarray}
\label{THMRxRcSP1PrEq11} \phi_3=b_2
\end{eqnarray}
for a constant $b_2$.

Next, by using \eqref{THMRxRcSP1PrEq7} and \eqref{THMRxRcSP1PrEq9} and considering \eqref{THMRxRcSP1PrEq11} we obtain
\begin{eqnarray}
\label{THMRxRcSP1PrEq12} \phi(u,v)&=&(-u\sinh\theta_0,\phi_1(u),\phi_2(u),b_2,b_3 \sin\frac{v}{b_3},b_3 \cos\frac{v}{b_3})
\end{eqnarray}
for a non-zero constant constant $b_2$, where we put $a=\frac 1{b_3}$.  \eqref{THMRxRcSP1PrEq10} turns into
\begin{eqnarray}
\label{THMRxRcSP1PrEq13} \phi_1'{}^2+\phi_2'{}^2=\cosh^2\theta_0
\end{eqnarray}
and from $\langle e_6,e_6\rangle=1$ we have
\begin{eqnarray}
\label{THMRxRcSP1PrEq14} \phi_1^2+\phi_2^2=b_0^2,
\end{eqnarray}
where we put $b_0=\sqrt{1-(b_2^2+b_3^2)}$. Because of \eqref{THMRxRcSP1PrEq13} and \eqref{THMRxRcSP1PrEq14}, up to a suitable rotation, we may assume
\begin{eqnarray}\label{THMRxRcSP1PrEq15}
\phi_1=b_0 \cos \left(\frac{u \cosh \theta _0 }{b_0}\right),&\qquad& \phi_2=b_0 \sin \left(\frac{u \cosh \theta _0 }{b_0}\right).
\end{eqnarray}
By combining \eqref{THMRxRcSP1PrEq12} and \eqref{THMRxRcSP1PrEq15}, we obtain

\begin{align}\label{THMRxRcSP1PrEq16}
\begin{split}
 \phi(u,v)=&\left(-u\sinh\theta_0,b_0 \cos \left(\frac{u \cosh \theta _0 }{b_0}\right),b_0 \sin \left(\frac{u \cosh \theta _0 }{b_0}\right),b_2,\right.\\
&\left.b_3 \sin\frac{v}{b_3},b_3 \cos\frac{v}{b_3}\right).
\end{split}
\end{align}
By a direct computation, we obtain the mean curvature vector $H$ of $M$ in $\mathbb E^1_1\times \mathbb S^4$ as
\begin{align}\label{THMRxRcSP1PrEq17}
\begin{split}
H=&\left(0,b_4 \cos \left(\frac{u \cosh \left(\theta _0\right)}{b_0}\right),b_4 \sin \left(\frac{u \cosh \left(\theta _0\right)}{b_0}\right),\frac{1}{2} b_2 \left(\cosh ^2\left(\theta _0\right)+1\right),\right.\\&\left.b_5 \sin \left(\frac{v}{b_3}\right),b_5 \cos \left(\frac{v}{b_3}\right)\right),
\end{split}
\end{align}
where we put $b_4=\frac{b_0^2 \left(\cosh ^2\left(\theta _0\right)+1\right)-\cosh ^2\left(\theta _0\right)}{2  b_0}$ and $ b_5 =\frac{b_3^2 \left(\cosh ^2\left(\theta _0\right)+1\right)-1}{2 b_3}$. Note that we have $\langle H,\eta\rangle=0$. By a further computation using \eqref{THMRxRcSP1PrEq1} and \eqref{THMRxRcSP1PrEq17}, we see that $H$ is parallel if and only if $b_4=0$ which implies
\begin{eqnarray}\label{THMRxRcSP1PrEq18}
b_0= \frac{\cosh \theta _0}{\sqrt{\cosh ^2\theta _0+1}}.
\end{eqnarray}
Finally, we combine \eqref{THMRxRcSP1PrEq16} with \eqref{THMRxRcSP1PrEq18} and put $b_1=\sinh\theta_0$ to get \eqref{THMRxRcSP1Eq1}. 

The proof of the converse follows from a direct computation.
\end{proof}


\section{Reduction of Co-dimension}
In this section, we consider the substantial co-dimension of a biconservative surfaces in Robertson-Walker spaces with higher dimension.

Let $M$ be an oriented space-like biconservative PMCV surface in $L^n_1(f,c), n\geq 6$. Consider the  positively oriented global vector fields $e_1,e_2,e_3,e_4$ and the function $\theta:M\to\mathbb R$ defined by \eqref{RWTS-etaandTDefNew} and $H=H_0 e_4$. Also, by considering Proposition \ref{PropNonDegeBicCond}, define the vector field $\xi$ on $M$ and a function $\gamma:M\to\mathbb R$ by
\begin{align}\label{CoDimEq1}
\begin{split}
 \langle\xi,\eta\rangle =0, \qquad& h(e_1,e_2)=0,\\
h(e_1,e_1)=-\gamma e_3+\xi, \qquad& h(e_2,e_2)=\gamma e_3+2H_0 e_4-\xi,
\end{split}
\end{align}
where we put $\gamma=\gamma_{e_3}=h^3_{11}.$

We are going to use the following lemma.
\begin{lemma}\label{CoDimLemma1}
Let $M$ be an oriented space-like biconservative PMCV surface in $L^n_1(f,c), n\geq 6$. Then, we have two cases:
\begin{enumerate}
\item [Case 1.] $\dim N_1M=2$ at every point of $M$, $\eta\in N_1M$ and $\nabla^\perp (N_1M)\subset N_1M$,
\item [Case 2.] $\dim N_2M=3$ at every point of $M$, $\eta\in N_2M-N_1M$  and $\nabla^\perp (N_2M)\subset N_2M$.
\end{enumerate}
\end{lemma}

\begin{proof}
Put $\langle\xi,\xi\rangle=\zeta$. First, we claim that either $\zeta$ is identically zero or it is non-vanishing on $M$. We combine \eqref{CorryofPropNonDegeBicCond1} and \eqref{CoDimEq1} with the equations \eqref{CodazziLn1fc1}, \eqref{CodazziLn1fc2}, \eqref{SpaceLikeTNablaXpdte1N} and \eqref{SpaceLikeTNablaXpdte2N}  to get
\begin{eqnarray}    
    \label{CoDimLemma1Eq1} e_2(\gamma)=e_2(\theta)&=&0,\\
    \label{CoDimLemma1Eq2} \nabla^\perp_{e_2}e_3=\nabla^\perp_{e_2}\xi&=&0,\\
    \label{CoDimLemma1Eq3b} \nabla^\perp_{e_1}e_3&=&-\tanh\theta\xi,\\
		\label{CoDimLemma1Eq3} \cosh\theta\sinh\theta\left(-\frac{f''}{f} + \frac{f'{}^2 + c}{f^2}\right)e_3 &=& e_1(\gamma)e_3-\gamma\tanh\theta\xi-\nabla^\perp_{e_1} \xi
\end{eqnarray}
from which we also have 
\begin{eqnarray}  
    \label{CoDimLemma1Eq4} e_2(\zeta)&=&0,\\
    \label{CoDimLemma1Eq5} e_1(\zeta)&=&2\gamma\tanh\theta\zeta.
\end{eqnarray}

Now, assume the existence of $p\in M$ such that $\xi(p)=0$. By considering \eqref{CorryofPropNonDegeBicCond1}, we use  a local coordinate system  $(\bar u,\bar v)$ on $\mathcal N_p'$ starting from $(u,v)$ such that $\left.e_1\right|_{\mathcal N_p'}=\partial_u$ and $\left.e_2\right|_{\mathcal N_p'}=\partial_v$. Then, because of \eqref{CoDimLemma1Eq1},  \eqref{CoDimLemma1Eq4} and  \eqref{CoDimLemma1Eq5}, we have 
$$\zeta=\zeta(u),\ \gamma=\gamma(u),\ \theta=\theta(u), \qquad \zeta'=2\gamma\tanh\theta\zeta, \ \zeta(0)=0$$
on $\mathcal N_p'$. This shows that $\zeta=0$ on $\mathcal N_p'$. By the connectedness, $\zeta=\langle\xi,\xi\rangle=0$ on $M$. This proves our claim. Since $\xi$ is space-like by \eqref{CoDimEq1}, we have either $\xi=0$ on $M$ or $\xi(p)\neq 0$ for all $p\in M$.

\textbf{Case 1}. $\xi=0$ on $M$. In this case, \eqref{CoDimLemma1Eq3b} and \eqref{CoDimLemma1Eq3} turn into 
\begin{eqnarray}    
    \label{CoDimLemma1Case1Eq1} \nabla^\perp_{e_1}e_3&=&0,\\
    \label{CoDimLemma1Case1Eq2} e_1(\gamma) &=& \cosh\theta\sinh\theta\left(-\frac{f''}{f} + \frac{f'{}^2 + c}{f^2}\right)
\end{eqnarray}
We note the first equation in \eqref{LemmaL41f0SPLEq5} and \eqref{CoDimLemma1Eq1} ensures that $\gamma$ is non-vanishing on $M$. So we have the Case 1 of the lemma because of \eqref{CoDimLemma1Eq2} and \eqref{CoDimLemma1Case1Eq1}.

\textbf{Case 2}. $\xi(p)\neq0$ for all $p\in M$. In this case, \eqref{CoDimEq1} implies $N_1=\mathrm{span\, } \{\gamma e_3-\xi,e_4\}$ and $\eta\not\in N_1$. Moreover, from \eqref{CoDimLemma1Eq3} we have
\begin{equation}    
   \label{CoDimLemma1Case2Eq1} \cosh\theta\sinh\theta\left(\frac{e_1(\gamma)}{\zeta}+\frac{f''}{f\zeta} - \frac{f'{}^2 + c}{\zeta f^2}\right)e_3 = \nabla^\perp_{e_1} \frac1{\zeta}\xi.
\end{equation}
\eqref{CoDimLemma1Eq2}, \eqref{CoDimLemma1Eq3b} and \eqref{CoDimLemma1Case2Eq1} yield the Case 2 of the lemma.
\end{proof}


Now, we are ready to prove the main result of this section.
\begin{theorem}\label{CoDimMainTheorem}
Let $M$ be an oriented space-like biconservative PMCV surface in $L^n_1(f,c), n\geq 6$. Then, there exists a totally geodesic submanifold $N$ of $L^n_1(f,c)$ such that $M\subset N$ and $\dim N$ is either 4 or 5.  
\end{theorem}

\begin{proof}
Let $p\in M$. We are going to consider the Case 1 and Case 2 of Lemma \ref{CoDimLemma1}, separately. Note that by Corollary \ref{CorryFlatNB}, $M$ has flat normal bundle.

\noindent\textbf{Case 1.} In this case, $N_1M$ and its orthogonal complementary $(N_1M)^\perp$ in the normal bundle of $M$ are invariant under the normal connection. Since $M$ has flat normal bundle, similar to the proof of \cite[Lemma 1]{Erbacher1971}, there exists orthonormal vector fields $\xi_5,\xi_6,\hdots,\xi_n\in (N_1M)^\perp$ such that $\nabla^\perp \xi_\alpha=0$ which yields $\widetilde\nabla \xi_\alpha=0$ for $\alpha\geq 5$.

Let $(\mathcal N_p,(u,v))$ be a local coordinate system near to $p$ described in Lemma \ref{CorryofPropNonDegeBicCond}.
 Put $\bar{\mathcal N}_p=\Pi^2(\mathcal N_p)$. First, we claim that $\bar{\mathcal N}_p$ lies on a 3-dimensional totally 
geodesic submanifold $\bar{N}_p$ of $R^{n-1}(c)$. Note that we have
\begin{eqnarray*}
0&=&\widetilde\nabla_{\partial_u} \xi_\alpha=\nabla^0_{\partial_u} \xi_\alpha+\frac{f'}{f}\xi_\alpha,\\
0&=&\widetilde\nabla_{\partial_v} \xi_\alpha=\nabla^0_{\partial_v} \xi_\alpha\\
\end{eqnarray*}
which imply that $$\xi_\alpha=\frac 1f C_\alpha$$  for $\alpha\geq 5$, where $C_\alpha$ satisfies
\begin{equation}\label{CoDimMainTheoremPEq1}
\nabla^0 C_\alpha=0. 
\end{equation}
 On the other hand, because $\eta\in N_1M$, we have $\widetilde g(C_\alpha,\eta)=0$ which implies
\begin{equation}\label{CoDimMainTheoremPEq0}
g_c(\Pi^2_* (C_\alpha),X)=0
\end{equation}
whenever $X$ is a vector field in $R^{n-1}(c)$ tangent to $\bar{\mathcal N}_p$
and we have 
\begin{equation}\label{CoDimMainTheoremPEq0b}
\widetilde g_c(\Pi^2_* (C_\alpha),\Pi^2_* (C_\beta))=\delta_{\alpha\beta}.
\end{equation}
Thus, \eqref{CoDimMainTheoremPEq0} and  \eqref{CoDimMainTheoremPEq0b} yield that $\{\Pi^2_* (C_5),\Pi^2_* (C_6),\hdots,\Pi^2_* (C_n)\}$ is an orthonormal set of vector fields normal to $\bar{\mathcal N}_p$ in $R^{n-1}(c)$. Furthermore, from \eqref{CoDimMainTheoremPEq1} we have 
$$\nabla^{R^{n-1}(c)}_X\Pi^2_* (C_\alpha)=0.$$
Therefore, \cite[Theorem]{Erbacher1971} implies that $\bar{\mathcal N}_p$ lies on a $3$-dimensional totally geodesic submanifolds $\bar{N}_p$ of $R^{n-1}(c)$. Put $N_p=I\times_f \bar{N}_p$. Then, $N_p$ is a $4$-dimensional totally geodesic submanifold of $L^n_1(f,c)$ and $\mathcal N_p\subset N_p$. This proves the local result.

The global result follows from being constant of $\dim N_1$ on $M$ similar to the idea of the proof of \cite[Proposition 1]{Erbacher1971}: If $q\in M$ and $\mathcal N_q\cap\mathcal  N_p\neq\varnothing$, then we have $\mathcal N_p\cap\mathcal  N_q\subset N_p\cap N_q$ for a totally geodesic submanifold $N_q\ni q$. If  $N_p\neq N_q$, then  $\mathrm{dim\, }N_p\cap N_q<4$ which implies  that $\mathrm{dim\, } N_1<2$ at $r$ whenever $r\in \mathcal N_q\cap\mathcal  N_p$.  However, this yields a contradiction.  Hence, we have $N_p= N_q$. By the connectedness, $M\subset N^4$ and $N$ is a totally geodesic submanifold of $L^n_1(f,c)$. 

\noindent\textbf{Case 2.} In this case, similar to Case 1, we consider orthonormal vector fields $\xi_6,\xi_7,\hdots,\xi_n\in (N_2M)^\perp$ which satisfies $\widetilde\nabla \xi_\alpha=0$ for $\alpha\geq 6$. We get \eqref{CoDimMainTheoremPEq1}, \eqref{CoDimMainTheoremPEq0} and \eqref{CoDimMainTheoremPEq0b} for $\alpha\geq 6$. Since $\dim N_2$ is constant on $M$, \eqref{CoDimMainTheoremPEq1}, \eqref{CoDimMainTheoremPEq0} and \eqref{CoDimMainTheoremPEq0b} for $\alpha\geq 6$  yield the same conclusion for a totally geodesic submanifold $N^5$ of $L^n_1(f,c)$. 
\end{proof}

By combining Theorem \ref{CoDimMainTheorem} with Remark \ref{NonExistenceRemark} and Proposition \ref{NonExistenceE11H4}, we get the following results:
\begin{corry}\label{CoDimMainTheoremCor2}
A space-like biconservative PMCV surface in  $\mathbb E^1_1\times \mathbb S^n,\ n\geq 6$ is contained in a totally geodesic submanifold $N^5$ of $\mathbb E^1_1\times \mathbb S^n$.
\end{corry}

\begin{corry}\label{CoDimMainTheoremCor1}
There are no space-like biconservative PMCV surface in  $\mathbb E^1_1\times \mathbb H^n$.
\end{corry}


\section*{Acknowledgements}
This work was carried out during the 1001 project supported by the Scientific and Technological Research Council of T\"urkiye (T\"UB\.ITAK)  (Project Number: 121F352).

\end{document}